\documentclass[a4paper,reqno,11pt]{amsart}

\usepackage[T1]{fontenc}

\usepackage{microtype}

\usepackage{lmodern}

\usepackage{amssymb}

\numberwithin{equation}{section}

\usepackage{hyperref}

\usepackage[abbrev, backrefs]{amsrefs}

\usepackage{enumerate}

\newtheorem{proposition}{Proposition}[section]
\newtheorem{theorem}{Theorem}
\newtheorem{lemma}[proposition]{Lemma}

\theoremstyle{definition}
\newtheorem{definition}{Definition}[section]

\theoremstyle{remark}

\newcommand{\abs}[1]{\lvert#1\rvert}
\newcommand{\bigabs}[1]{\bigl\lvert#1\bigr\rvert}
\newcommand{\Bigabs}[1]{\Bigl\lvert#1\Bigr\rvert}

\newcommand{\st}{\: :\:}
\newcommand{\norm}[1]{\lVert#1\rVert}
\newcommand{\R}{\mathbf{R}}
\newcommand{\dif}{\,\mathrm{d}}
\newcommand{\C}{\mathbf{C}}
\newcommand{\Lin}{\mathcal{L}}
\newcommand{\N}{\mathbf{N}}

\DeclareMathOperator{\supp}{supp}
\DeclareMathOperator{\Div}{div}

\DeclareMathOperator{\id}{id}

\DeclareMathOperator{\rank}{rank}

\title[Hardy--Sobolev inequalities and canceling operators]{Hardy--Sobolev inequalities for vector fields and canceling linear differential operators}

\author{Pierre Bousquet}
\address{Laboratoire d'analyse, topologie, probabilit\'es UMR7353\\
Aix-Marseille Universit\'e\\
CMI 39\\
Rue Fr\'ed\'eric Joliot Curie\\
13453 Marseille Cedex 13\\
France}
\email{bousquet@cmi.univ-mrs.fr}

\author{Jean Van Schaftingen}
\address{Universit\'e catholique de Louvain\\
Institut de Recherche en Math\'ematique et Physique (IRMP)\\
Chemin du Cyclotron 2 bte L7.01.01\\
1348 Louvain-la-Neuve\\
Belgium}
\email{Jean.VanSchaftingen@uclouvain.be}
\date{\today}
\begin{document}

\begin{abstract}
Given  a homogeneous \(k\)-th order differential operator \(A (D)\) on \(\R^n\) between two finite dimensional spaces, we establish the Hardy inequality \[\int_{\R^n} \frac{\abs{D^{k-1}u (x)}}{\abs{x}} \dif x \leq C \int_{\R^n} \abs{A(D)u}\] and the Sobolev inequality \[\norm{ D^{k-n} u }_{L^{\infty}(\R^n)}\leq C \int_{\R^n} \abs{A(D)u}\]  when \(A(D)\) is elliptic and satisfies a recently introduced cancellation property. We recover in particular a Hardy inequality due 
to V.\thinspace{}Maz\cprime{}ya, and a Sobolev inequality due to J.\thinspace{}Bourgain and H.\thinspace{}Brezis. We also study the necessity of these two conditions.
\end{abstract}

\maketitle

\tableofcontents

\section{Introduction}
Let \(k\in \N_*\) and let \(V, E\) be two real finite dimensional vector spaces. Given  a homogeneous \(k\)-th order differential operator \(A (D)\) on \(\R^n\) from  \(V\) to \(E\), we address the question of controlling any vector field \(u  \in C^\infty (\R^n; V)\)  by the vector field \(A (D) u\), where the vector differential operator \(A (D)\) is defined by
\[
 A (D) u  = \sum_{\substack{\alpha \in \N^n\\ \abs{\alpha} = k}} A_\alpha \bigl[\partial^\alpha u \bigr] \in C^\infty (\R^n; E).
\]
Here, \(A_\alpha\) is a linear map in \(\Lin (V; E)\), for every \(\alpha \in \N^n\) with \(\abs{\alpha} = k\).

The classical theory of A.\thinspace P.\thinspace Calder\'on and A.\thinspace Zygmund asserts that for every \(p \in (1, \infty)\), there exists \(C > 0\) such that for each compactly supported smooth vector field \(u  \in C^\infty_c (\R^n; V)\),
\[
 \int_{\R^n} \abs{D^{k} u }^p
 \le C  \int_{\R^n} \abs{A (D) u }^p 
\]
if and only if the operator \(A (D) \) is \emph{elliptic} \cite{CZ1952}, that is for every \(\xi \in \R^n \setminus \{0\}\), the linear map
\(A (\xi) := \sum_{\alpha \in \N^n, \abs{\alpha} = k} \xi^\alpha A_\alpha \in \Lin (V; E)\) is one-to-one \citelist{\cite{Hormander1958}*{theorem 1}\cite{Spencer}*{definition 1.7.1}\cite{Agmon1959}*{\S 7}\cite{Agmon}*{definition 6.3}}.
Examples of first-order homogeneous elliptic operators are given for \(V = \R\) by the gradient operator \(A (D) u  = \nabla u \) and for \(V = \bigwedge^\ell \R^n\) by the exterior differential and codifferential \(A (D) u  = (du , d^* u )\).

The situation is dramatically different for \(p = 1\) as there is no nontrivial estimate of the \(L^1\)-norm of some component of \(D^k u \) by \(\int_{\R^n} \abs{A (D) u}\) \citelist{\cite{Ornstein1962}\cite{CKCRAS}\cite{CKPaper}}.
This does not end the story however. Even if the quantity \(\int_{\R^n} \abs{A (D) u }\) is strictly weaker than \(\int_{\R^n} \abs{D u }\), it might still be possible to replace the latter by the former in some inequalities. 

A first inequality to which this programme was applied is the \emph{Gagliardo--Nirenberg--Sobolev inequality} \citelist{\cite{Gagliardo}\cite{Nirenberg1959}}
\begin{equation}
 \label{eqGNS}
  \Bigl(\int_{\R^n} \abs{u }^\frac{n}{n - 1}\Bigr)^{1 - \frac{1}{n}}
 \le C \int_{\R^n} \abs{D u }.
\end{equation}
The elliptic operators that can replace the derivative were characterized by a new cancellation condition.

\begin{theorem}[Van Schaftingen \cite{VSVectorL1}]
\label{theoremCancelingSobolev}
Let \(A(D)\) be an elliptic homogeneous linear differential operator of order
\(k\) on \(\R^n\) from \(V\) to \(E\) and \(\ell \in \{1, \dotsc, \min( k ,\\ n - 1)\}\).  The estimate
\[
  \Bigl(\int_{\R^n} \abs{D^{k - \ell}u }^\frac{n}{n - \ell} \Bigr)^{1-\frac{\ell}{n}} \le C \int_{\R^n} \abs{A (D) u},
\]
holds for every \(u \in C^\infty_c(\R^n; V)\)
if and only if \(A(D)\) is canceling.
\end{theorem}

The new cancellation condition was defined as

\begin{definition}
A homogeneous linear differential operator \(A(D)\) on \(\R^n\) from \(V\) to \(E\) is \emph{canceling} if 
\[
  \bigcap_{\xi \in \R^n \setminus\{0\}} A(\xi)[V]=\{0\}.
\]
\end{definition}

Theorem~\ref{theoremCancelingSobolev} covers in particular the classical inequality \eqref{eqGNS}, the Hodge--Sobolev inequality of J.\thinspace Bourgain and H.\thinspace Brezis, and L.\thinspace Lanzani and E.\thinspace Stein \citelist{\cite{BB2004}\cite{BB2007}\cite{LanzaniStein2005}} (see also \citelist{\cite{BB2002}\cite{BB2003}\cite{BourgainBrezisMironescu}\cite{Mazya2007}\cite{VS2004BBM}\cite{VS2004Divf}}) and 
the Korn--Sobolev inequality \cite{Strauss1973}.

In the present work we continue this programme for other classical inequalities in Sobolev spaces.
We begin with the classical \emph{Hardy inequality}: given \(n \geq 2\) and \(k \geq 1\), there exists \(C>0\) such that for every \(u\in C^{\infty}_c(\R^n ; V)\),
\begin{equation}
\label{eqClassicalHardy}
 \int_{\R^n} \frac{\abs{D^{k-1}u (x)}}{\abs{x}}\dif x \le C \int_{\R^n} \abs{D^k u};
\end{equation}
and we address the validity of the following inequality
\begin{equation}\label{hardyinequalityintroduction}
  \int_{\R^n} \frac{\abs{D^{k - 1}u (x)}}{\abs{x}}\dif x \le C \int_{\R^n} \abs{A (D) u}.
\end{equation}
Remarkably, it also depends on the cancellation condition.

\begin{theorem}
\label{theoremCancelingHardy}
Let \(A(D)\) be an elliptic homogeneous linear differential operator of order
\(k\) on \(\R^n\) from \(V\) to \(E\) and \(\ell \in \{1, \dotsc, \min( k  , n - 1)\}\). The estimate
\begin{equation*}
  \int_{\R^n} \frac{\abs{D^{k - \ell}u (x)}}{\abs{x}^\ell}\dif x \le C \int_{\R^n} \abs{A (D) u},
\end{equation*}
holds for every \(u \in C^\infty_c(\R^n; V)\) if and only if \(A(D)\) is canceling.
\end{theorem}

As particular cases of theorem~\ref{theoremCancelingHardy}, we have the classical Hardy inequality \eqref{eqClassicalHardy}, the inequality of V.\thinspace Maz\cprime{}ya \citelist{\cite{Mazya2010}\cite{BousquetMironescu}}: for every \(u \in C^\infty_c (\R^n;\R^n)\),
\begin{equation}
 \label{ineqMazya}
\int_{\R^n} \frac{\abs{D u (x)}}{\abs{x}}\dif x \le C \int_{\R^n} \abs{\Delta u} + \abs{\nabla \Div u},
\end{equation}
a new Hodge--Hardy inequality: for every \(u \in C^\infty_c (\R^n; \bigwedge^{\ell} \R^n)\), if \(2 \le \ell \le n - 2\), 
\[
 \int_{\R^n} \frac{\abs{u (x)}}{\abs{x}}\dif x \le C \int_{\R^n} \abs{d u} + \abs{d^* u},
\]
and a new Korn--Hardy inequality: for every \(u \in C^\infty_c(\R^n)\),
\[
 \int_{\R^n} \frac{\abs{u (x)}}{\abs{x}}\dif x \le C \int_{\R^n} \abs{\nabla^s u},
\]
where the symmetric derivative is defined as \(\nabla^s u (x) = \frac{1}{2}\bigl(D u(x) + (D u (x))^*\bigr)\).

The proof of the sufficiency of the cancellation in theorem~\ref{theoremCancelingHardy} is quite different from its counterpart in theorem~\ref{theoremCancelingSobolev}. It combines in an original way the strategy of Bousquet and Mironescu \cite{BousquetMironescu} with algebraic properties of canceling operators  \cite{VSVectorL1} and properties of Green functions \cite{HormanderI}.

To conclude our discussion of Hardy inequalities, we would like to mention that H.\thinspace Castro, J.\thinspace  D{\'a}vila and Wang Hui have recently obtained another family of unexpected Hardy inequalities \citelist{\cite{CastroWang2010}\cite{CastroDavilaWang2011}\cite{CastroDavilaWang2013}}. 
Whereas we are concerned with point singularities, they are concerned with boundary singularities and obtain inequalities of the form
\[
  \int_{\R^{n - 1} \times \R^+} \Bigabs{ D \Bigl( \frac{u (y)}{y_n}\Bigr)}\,dy
  \le C \int_{\R^{n - 1} \times \R^+} \abs{ D^2 u};
\]
cancellations also play a crucial role in their proofs.

The second inequality that we study is the limiting Sobolev inequality (see for example \cite{Brezis2011}*{chapter 9, remark 13})
\[
 \sup_{x \in \R^n} \abs{u (x)} \le C \int_{\R^n} \abs{D^n u}.
\]
We prove a limiting case of theorem~\ref{theoremCancelingSobolev} that was left open \cite{VSVectorL1}*{open problem 8.4}. Again, the cancelation property plays a role.

\begin{theorem}
\label{theoremLinfty}
Let \(A(D)\) be a homogeneous linear differential operator of order
\(k \ge n\) on \(\R^n\) from \(V\) to \(E\). If \(A (D)\) is elliptic and canceling, then the estimate
\[
  \sup_{x \in \R^n} \abs{D^{k-n}u (x)} \le C \int_{\R^n} \abs{A (D) u},
\]
holds for every \(u \in C^\infty_c(\R^n; V)\).
\end{theorem}

Theorem~\ref{theoremLinfty} covers in particular the estimate of P.\thinspace Mironescu
\(
 \norm{\nabla u}_{L^{\infty}} \le C \norm{\Delta^{\frac{n}{2}} \nabla u}_{L^1}
\)
for \(n \in \N\) odd \cite{Mironescu2010}. The proof of theorem~\ref{theoremLinfty} relies on theorem~\ref{theoremCancelingHardy}.

The cancellation is not necessary for the estimate of theorem~\ref{theoremLinfty} to hold \cite{VSVectorL1}*{remark 5.1}: for example, the differentiation operator on \(\R\) is \emph{not canceling}, but the inequality \(\norm{u}_{L^\infty} \le \norm{u'}_{L^1}\) holds for every \(u \in C^\infty_c (\R)\).

In the scalar case \(\dim V=1\), the inequalities of theorems~\ref{theoremCancelingSobolev},~\ref{theoremCancelingHardy} and~\ref{theoremLinfty} follow from the Sobolev embedding of \(W^{1,1} (\R^n)\) in the Lorentz space \(L^{\frac{n}{n - 1}, 1} (\R^n)\) \citelist{\cite{Alvino1977}\cite{Tartar1998}}
\[
  \norm{u}_{L^{\frac{n}{n -1}, 1}(\R^n)} \le C \norm{D u}_{L^1 (\R^n)}.
\]
It is not known whether this inequality can be extended to canceling operators \cite{VSVectorL1}*{open problem 8.3} as
\[
  \norm{u}_{L^{\frac{n}{n -1}, 1}(\R^n)} \le C \norm{A(D) u}_{L^1 (\R^n)};
\]
this inequality would be consistent with theorems~\ref{theoremCancelingSobolev},~\ref{theoremCancelingHardy} and~\ref{theoremLinfty}.

One can wonder whether the \emph{ellipticity} is \emph{necessary} in theorem~\ref{theoremCancelingHardy} as it is in theorem~\ref{theoremCancelingSobolev}  \(\ell = 1\) \cite{VSVectorL1}*{proposition 5.1}. In general, this is not the case. However, when \(\ell=1\), the ellipticity is necessary for a scale of Hardy-Sobolev inequalities.

\begin{theorem}
\label{theoremCancelingHardySobolev}
Let \(A(D)\) be a homogeneous linear differential operator of order
\(k\) on \(\R^n\) from \(V\) to \(E\) and let \(\lambda \in [0, 1)\). The estimate
\[
  \Bigl(\int_{\R^n} \frac{\abs{D^{k - 1}u }^\frac{n - \lambda}{n - 1}}{\abs{x}^\lambda}\dif x \Bigr)^\frac{n - 1}{n - \lambda} \le C \int_{\R^n} \abs{A (D) u},
\]
holds for every \(u \in C^\infty_c(\R^n; V)\)
if and only if \(A(D)\) is elliptic and canceling. 
\end{theorem}

For \(A (D) = (\Delta, \nabla \Div)\), we recover an inequality of Maz\cprime{}ya \cite{Mazya2010} (see also \cite{BousquetMironescu}). For general \(A (D)\), this result is already known for \(\lambda=0\) \cite{VSVectorL1}. For \(\lambda \in (0, 1)\),
the sufficiency part is a consequence of theorems~\ref{theoremCancelingSobolev} and~\ref{theoremCancelingHardy} by the H\"older inequality. 
Alternatively, it can be proved in a more direct way by using the same arguments as its counterpart in theorem~\ref{theoremCancelingHardy} bypassing the more delicate proof of theorem~\ref{theoremCancelingSobolev}.

In the limiting case \(\lambda=1\) in theorem~\ref{theoremCancelingHardySobolev}, the ellipticity condition is not necessary in theorem~\ref{theoremCancelingHardy}. This phenomenon can already be observed in the scalar case: for every \(u \in C^\infty_c (\R^3)\), one has 
\begin{equation}\label{contrexempleellipticite}
 \int_{\R^3} \frac{\abs{u (x)}}{\abs{x}}\dif x \le C \int_{\R^3} \abs{\partial_1 u} + \abs{\partial_2 u},
\end{equation}
and the operator \((\partial_1, \partial_2)\) is not elliptic.

We give some partial results concerning the Hardy inequality \eqref{hardyinequalityintroduction} when the operator \(A(D)\) is not elliptic. First the cone \[\{\xi \in \R^n \st A(\xi) \text{ is not one-to-one}\}\] should not be too large: for example, it cannot contain a hyperplane.
In particular, the ellipticity condition turns out to be necessary when \(n = 2\).
The general problem of writing necessary and sufficient conditions on \(A (D)\) seems quite difficult, we have however written in theorem~\ref{propositionHardyRankOne} such a condition for an operator \(A (D)\) which is a collection of components of first order derivatives:
\[
 A (D) u (x) = (a_1 \cdot D u (x)[b_1],\dotsc, a_\ell \cdot D u (x) [b_\ell]).
\]


\section{Proof of the Hardy--Sobolev inequality}

The first tool that we shall use is the existence of a Green function for \(A (D)\).

\begin{lemma}
\label{lemmaGreen}
Let \(A(D)\) be a linear differential operator of order \(k\) on \(\R^n\) from \(V\) to \(E\).
If \(A (D)\) is elliptic then there exists a function \(G \in C^\infty \bigl(\R^n \setminus \{0\}; \mathcal{L}(E ; V)\bigr) \cap L^1_\mathrm{loc} \bigl(\R^n; \mathcal{L}(E ; V)\bigr)\) such that for every \(u \in C^\infty_c(\R^n; V)\) and \(x \in \R^n\),
\[
  u (x) = \int_{\R^n} G (x - y) \bigl[A (D) u (y)\bigr] \dif y\:.
\]
Moreover, for every \(\ell \in \N\) such that \(\ell > k-n\), \(D^{\ell} G\) is homogeneous of degree \(k-n-\ell\) and 
\[
 \bigcap_{x \in \R^n \setminus \{0\}} \ker G (x) = \bigcap_{\xi \in \R^n \setminus \{0\}} \ker A (\xi)^*.
\]
\end{lemma}

Here and in the sequel, we endow \(V\) and \(E\) with an inner product denoted by \(\cdot\) and the adjoint is taken with respect to that fixed Euclidean structure.

The restriction \(\ell > k-n\) is essential as it can be observed  when \(n = k = 2\) and \(A (D) = \Delta\). In that case, the Green function \(G\) is not homogeneous of degree \(0\).

We use the following convention to define the Fourier transform of a map \(u\in C^{\infty}_c(\R^n ; V)\):
\[
\widehat{u}(\xi) = \int_{\R^n} u(x)e^{-i \xi \cdot x}\dif x\:;
\]
the map \(\widehat{u}\) takes its values in the \textit{complexified} vector space \(V \otimes \C\).

\begin{proof}[Proof of lemma~\ref{lemmaGreen}]
We define the map \( H : \R^n \setminus \{0\} \to \mathcal{L}(E ; V)\), for every \(\xi \in \R^n \setminus \{0\}\) by
\[
H(\xi) = \bigl( A (\xi)^* \circ A (\xi) \bigr)^{-1} \circ A (\xi)^*.
\]
The map \(H\) is smooth in \(\R^n \setminus \{0\}\) and is homogeneous of degree \(-k\). If \(-k>-n\), then \(H\) belongs to \(L^{1}_\mathrm{loc}(\R^n)\) and defines a  distribution on \(\R^n\). If \(-k\leq -n\), then \(H\) can  be extended as  a distribution on \(\R^n\), still denoted by \(H\), in such a way that for every homogeneous polynomial \(P\) of degree \(\ell > k - n\), \(PH\) is a homogeneous distribution of degree \(\ell - k\) on \(\R^n\) (see for example \cite{HormanderI}*{theorem 3.2.4}). In both cases, \(H\) is a temperate distribution and \(\widehat{H}\in C^{\infty}(\R^n\setminus \{0\})\) \cite{HormanderI}*{theorem 7.1.18}.  

We define the temperate distribution \(G\) on \(\R^n\) by \(\widehat{G} = i^{-k} H\). For every \(\ell \in \N\) such that \(\ell > k - n\), \(\widehat{D^\ell G} = i^{\ell-k} \xi^{\otimes \ell} H\) is homogeneous of degree  \(\ell-k\). Hence, \(D^\ell G\) is homogeneous of degree  \(-n -\ell  + k\) \cite{HormanderI}*{theorem 7.1.16}. Since \(D^{k-n+1} G\) is homogeneous of degree \(-1>-n\), it is locally summable. This implies that \(G\) is locally summable as well: \(G\in L^1_\mathrm{loc} \bigl(\R^n; \mathcal{L}(E ; V)\bigr)\). Finally, for every \(\xi \in \R^n \setminus \{0\}\), \(\widehat{(A(D)u)}(\xi)=i^k A(\xi)\widehat{u}(\xi)\). Hence, \(G\) satisfies the required identity by  definition of  \(H\).
\end{proof}

The proof actually shows that \(D^\ell G\) is locally summable for every \(\ell \in \{0, \dotsc, k -1\}\). Hence, for these values of \(\ell\), the identity \(D^\ell u = (D^\ell G) * [A(D)u]\) which always  holds true in the sense of distributions, can be written in the following form:
\[
D^\ell u(x)=\int_{\R^n} D^\ell G(x-y)[A(D)u(y)]\, dy.
\]

The second ingredient that will be used repeatedly is a duality estimate on \(A (D) u\). 

\begin{lemma}
\label{lemmaCompensation}
Let \(A(D)\) be a linear differential operator of order \(k\) on \(\R^n\) from \(V\) to \(E\).
If \(A(D)\) is elliptic and canceling, then there exists \(C \in \R\) and \(m \in \N_*\) such that for every \(u \in C^\infty_c(\R^n; V)\) and every \(\varphi \in C^m (\R^n \setminus \{0\}; E)\) that satisfies for every \(j \in \{0, \dotsc, m \}\), 
\(
\abs{x}^j\abs{D^j\varphi}\in L^{1}_{\mathrm{loc}}(\R^n)
\),
\[
  \Bigl\lvert \int_{\R^n} \varphi \cdot A(D)u \Bigr\rvert \le
C \sum_{j = 1}^m \int_{\R^n} \abs{A(D)u(x)} \,\abs{x}^{j} \abs{D^j \varphi (x)}\dif  x\: .
\]
\end{lemma}
The integer \(m\) that appears in the conclusion depends on \(A(D)\) and not only on its order; a rather pessimistic upper bound for \(m\) is \(2k \dim V\) \cite{VSVectorL1}*{remark 4.1}.

The proof of this lemma is similar to the proof of  \cite{VSVectorL1}*{proposition 8.9}. The idea of the integration by parts already appeared in the context of divergence-free vector fields \citelist{\cite{VS2006BMO}*{lemma 4.5}\cite{BB2007}*{theorem 3}\cite{Mazya2010}*{theorem 2}\cite{BousquetMironescu}}. 

\begin{proof}[Proof of lemma~\ref{lemmaCompensation}]
Since \(A(D)\) is canceling and elliptic there exist a finite dimensional vector space \(F\) and  a homogeneous linear differential operator \(L(D)\) of order \(m \in \N_*\) from \(E\) to \(F\) such that \(L(D) \circ A(D) = 0\) and \(L(D)\) is cocanceling \cite{VSVectorL1}*{proposition 4.2}, that is,
\[
  \bigcap_{\xi \in \R^n \setminus \{0\}} \ker L (\xi) = \{0\}\:.
\]
Writing \(L(D) = \sum_{\abs{\alpha} = m} L_\alpha \partial^\alpha\), by classical properties of linear operators, there exist \(K_\alpha \in \Lin (F; E)\) for \(\alpha \in \N^n\) with \(\abs{\alpha} = m\) such that 
\begin{equation}
\label{eqKL}
  \sum_{\abs{\alpha} = m} K_\alpha \circ L_\alpha = \id_E
\end{equation}
(a detailed proof has been given in \cite{VSVectorL1}*{lemma 2.5}).
We define now, the homogeneous polynomial \(P : \R^n \to \Lin (E; F)\) of degree \(m\) for \(x \in \R^n\) by
\[
  P (x) = \sum_{\abs{\alpha} = m} \frac{ x^\alpha }{\alpha !} K_\alpha^*\:.
\]
By the identity \eqref{eqKL}, we compute
\[
 L(D)^* (P)= \sum_{\abs{\alpha} = m} L_\alpha{}^* \circ \partial^\alpha P = \sum_{\abs{\alpha} = m} L_\alpha{}^* \circ K_\alpha{}^*  = \id_E{}^* = \id_E.
\]

Let \(\varphi \in C^m (\R^n \setminus \{0\}; E)\) such that for every \(j \in \{0, \dotsc, m \}\), 
\(\abs{x}^j\abs{D^j\varphi}\in L^{1}_{loc}(\R^n).\) Since 
\[
\abs{L(D)^*(P[\varphi])} \leq C \sum_{j=0}^{m} \abs{x}^j\abs{D^j\varphi(x)},
\] 
this implies \(L(D)^*(P[\varphi])\in L^{1}_{\textrm{loc}}(\R^n ; E)\).
For every \(u \in C^\infty_c(\R^n; V)\) since \(L(D)[A(D)u] = 0\), we get by integration by parts
\[
\begin{split}
 \int_{\R^n} \varphi \cdot A(D)u &= \int_{\R^n} A(D)u \cdot (L(D)^*(P)) [\varphi]\\
&=\int_{\R^n} A(D)u \cdot \bigl( (L(D)^* (P)) [\varphi] - L(D)^*(P[\varphi])\bigr)\:.
\end{split}
\]
In order to conclude, we note that there exists \(C > 0\) such that for every \(x \in \R^n\),
\[
  \bigabs{\bigl( (L(D)^* P) [\varphi(x)] - L(D)^*(P[\varphi])(x)\bigr)} \le
C\Bigl(\sum_{j = 1}^m \abs{x}^{j} \abs{D^j \varphi (x)}\Bigr)\:.\qedhere
\]
\end{proof}

\begin{proof}[Proof of theorem~\ref{theoremLinfty}]
If \(G\) is the Green function given by lemma~\ref{lemmaGreen}, we have for every \(x \in \R^n\),
\[
 D^{k-n} u (x) = \int_{\R^n} D^{k-n} G (x - y)\bigl[A (D)u (y)\bigr]\dif y.
\]
Since \(D^{k-n+j} G\) is homogeneous of degree \(-j\) for every \(j \in \N_*\), we conclude by lemma \ref{lemmaCompensation} that 
\[
 \abs{D^{k-n} u (x)} \le C \int_{\R^n} \abs{A (D)u}.\qedhere
\]
\end{proof}

We prove with the same tools the sufficiency part of theorem~\ref{theoremCancelingHardy} and theorem~\ref{theoremCancelingHardySobolev}. The following proposition gives in fact a more general result:

\begin{proposition}
\label{propositionSufficientHardy}
Let \(A(D)\) be a linear differential operator of order \(k\) on \(\R^n\) from \(V\) to \(E\) and \(\ell \in \{1, \dotsc, \min(k,  n - 1) \}\).
If \(A(D)\) is elliptic and canceling, then for every \(u \in C^\infty_c(\R^n; V)\) and \(q \in [1, \frac{n}{n - \ell})\), there exists \(C\in \R\) such that
\[
  \Bigl( \int_{\R^n} \frac{\abs{D^{k - \ell} u (x)}^q}{\abs{x}^{n - (n - \ell) q}}\dif x \Bigr)^{\frac{1}{q}} \le C\int_{\R^n} \abs{A(D)u}.
\]
\end{proposition}

\begin{proof}
Let \(G\) be the Green function given by lemma~\ref{lemmaGreen}. We choose \(\rho \in C^\infty_c(\R^n)\) such that \(\rho = 1\) on \(B_{1/4}\) and \(\supp \rho \subset B_{1/2}\), and we define the kernels \(H\) and \(K\) for \(x, y \in (\R^n \setminus \{0\})  \times \R^n \) with \(x \ne y\) by
\[
 H (x, y) = \rho \Bigl( \frac{y}{\abs{x}} \Bigr) D^{k - \ell} G (x)
\]
and
\[
 K (x, y) = D^{k - \ell} G(x - y) - \rho \Bigl( \frac{y}{\abs{x}} \Bigr) D^{k - \ell} G (x)\:.
\]

By lemma~\ref{lemmaCompensation} and the homogeneity of \(D^{k - \ell} G\), we have
\begin{multline*}
 \biggl(\int_{\R^n} \Bigl\lvert\int_{\R^n}  H (x, y) [A (D) u (y)] \dif y\Bigr\rvert^q \frac{1}{\abs{x}^{n - (n - \ell) q}} \dif x \biggr)^\frac{1}{q}\\
=\biggl(\int_{\R^n} \Bigl\lvert\int_{\R^n}  \rho \Bigl( \frac{y}{\abs{x}} \Bigr) A (D) u (y) \dif y\Bigr\rvert^q \frac{\abs{D^{k - \ell} G(x)}^q}{\abs{x}^{n - (n - \ell) q}} \dif x \biggr)^\frac{1}{q}\\
\le C \biggl(\int_{\R^n} \Bigl\lvert\int_{B_{\abs{x}/2}} \frac{\abs{y}}{\abs{x}} \abs{A (D) u (y)} \dif y\Bigr\rvert^q \frac{1}{\abs{x}^{n}} \dif x \biggr)^\frac{1}{q}.
\end{multline*}
By the Minkowski inequality (see for example \cite{LiebLoss}*{theorem 2.4}), we get
\begin{multline*}
\biggl(\int_{\R^n} \Bigl\lvert\int_{B_{\abs{x}/2}} \frac{\abs{y}}{\abs{x}} \abs{A (D) u (y)} \dif y\Bigr\rvert^q \frac{1}{\abs{x}^{n}} \dif x \biggr)^\frac{1}{q}\\
\le \int_{\R^n} \abs{y} \abs{A (D) u (y)}\Bigl(\int_{\R^n  \setminus B_{2 \abs{y}}} \frac{1}{\abs{x}^{n + q}} \dif x \biggr)^\frac{1}{q}\dif y
\le C' \int_{\R^n} \abs{A (D) u}\:.
\end{multline*}

For the kernel \(K\), by the Minkowski inequality again, 
\begin{multline*}
 \biggl(\int_{\R^n} \Bigl\lvert\int_{\R^n}  K (x, y) \bigl[A (D) u (y)\bigr] \dif y\Bigr\rvert^q \frac{1}{\abs{x}^{n - (n - \ell) q}} \dif x \biggr)^\frac{1}{q}\\
 \le  \int_{\R^n} \Bigl(\int_{\R^n}  \frac{\abs{K (x, y)}^q}{\abs{x}^{n - (n - \ell) q}} \dif x\Bigr)^\frac{1}{q} \abs{A (D) u (y)}\dif y.
\end{multline*}
If \(x\not=0\), since \(K(x, \cdot)\) is continuously differentiable on \(B_{\abs{x}/2}\) and \(D^{k - \ell} G\) is homogeneous of degree \(-(n - \ell)\), if \(\abs{x} \ge 2 \abs{y}\)
\[
 \abs{K (x, y)} \le C\frac{\abs{y}}{\abs{x}^{n - \ell + 1} }\:;
\]
while if \(\abs{x} < 2 \abs{y}\),
\[
 \abs{K (x, y)} \le C\frac{1}{\abs{x-y}^{n - \ell} }\:.
\]
Therefore, since \(q < \frac{n}{n - \ell}\),
\begin{multline*}
 \int_{\R^n}  \frac{\abs{K (x, y)}^q}{\abs{x}^{n - (n - \ell) q}} \dif x\\
\le C \Bigl(\int_{B_{2 \abs{y}}} \frac{1}{\abs{x - y}^{(n - \ell) q}\abs{x}^{n - (n - \ell) q}} \dif x  + \int_{\R^n \setminus B_{2 \abs{y}}} \ \frac{\abs{y}^q}{\abs{x}^{n + q}} \dif x\Bigr) \le C'\:.
\end{multline*}
We conclude that 
\begin{equation*}
 \biggl(\int_{\R^n} \Bigl\lvert\int_{\R^n} K (x, y) \bigl[A (D) u (y)\bigr] \dif y\Bigr\rvert^q \frac{1}{\abs{x}^{n - (n - \ell) q}} \dif x \biggr)^\frac{1}{q} \le C'\int_{\R^n} \abs{A (D) u }\:.
\end{equation*}
This completes the proof of the proposition.
\end{proof}

We end this section with an alternate  proof of theorem~\ref{theoremLinfty}.

\begin{proof}[Proof of theorem~\ref{theoremLinfty} by proposition~\ref{propositionSufficientHardy}]
By proposition~\ref{propositionSufficientHardy}, there exists \(C>0\) such that for every \(u\in C^{\infty}_c(\R^n ; V)\),
\[
 \int_{\R^n} \frac{\abs{D^{k - n + 1} u (x)}}{\abs{x}^{n-1}}\dif x \le C \int_{\R^n} \abs{A (D) u}.
\]
On the other hand, we have the classical estimate (see for example \cite{Stein1970}*{\S 2.3 (18)})
\begin{equation}\label{eqlinftyder}
  \abs{D^{k - n} u(0)} \le C'\int_{\R^n} \frac{\abs{D^{k - n + 1} u (x)}}{\abs{x}^{n-1}}\dif x,
\end{equation}
which follows from the integration over \(\theta\in \mathbf{S}^{n-1}\), of the inequality
\[
\abs{D^{k-n}u(0)}= \Bigabs{\int_{\R^+} \frac{d}{dr}\bigl( D^{k-n}u( r \theta)\bigr) \dif r}\leq \int_{\R^+}\abs{D^{k-n+1}u(r\theta)} \dif r.
\]
Therefore, we have
\[
  \abs{D^{k - n} u(0)} \le C'' \int_{\R^n} \abs{A (D) u}.
\]
Since this estimate is invariant under translation, the conclusion follows.
\end{proof}

\section{Necessity of the cancellation condition}

In this section, we prove that if the Hardy inequality holds true for an elliptic operator \(A(D)\), then  \(A(D)\) is canceling.

\begin{proposition}
\label{propositionNecessaryCanceling}
Let \(A(D)\) be a linear differential operator of order \(k\) on \(\R^n\) from \(V\) to \(E\), \(\ell \in \{1, \dotsc, \min (k, n - 1)\}\) and \(q \in [1, \frac{n}{n - \ell}]\).
If \(A(D)\) is elliptic and if  there exists \(C \in \R\) such that  for every \(u \in C^\infty_c(\R^n; V)\)
\begin{equation*}
  \Bigl( \int_{\R^n} \frac{\abs{D^{k - \ell} u (x)}^q}{\abs{x}^{n - (n - \ell) q}}\dif x \Bigr)^\frac{1}{q} \le C\int_{\R^n} \abs{A(D)u}\:,
\end{equation*}
then \(A (D)\) is cancelling.
\end{proposition}
\begin{proof}
The proof follows the lines of the counterpart of the proposition for the Sobolev inequality \cite{VSVectorL1}*{proposition 5.5}. 
Let \(e\in \bigcap_{\xi \in \R^n \setminus \{0\}} A (\xi)[V] \). 
Let \(\psi\) be in the Schwartz class \(\mathcal{S}(\R^n)\) of rapidly decaying smooth functions be such that \(\widehat{\psi}=1\) on a neighborhood of \(0\) and define the family \((\rho_{\lambda})_{\lambda\geq 1}\) in \(\mathcal{S}(\R^n)\) for \(\lambda \ge 1\) and \(x \in \R^n\) by
\[
\rho_{\lambda}(x)=\lambda^n \psi(\lambda x)-\frac{1}{\lambda^n}\psi\Bigl(\frac{x}{\lambda}\Bigr).
\]
The family \((\rho_{\lambda})_{\lambda\geq 1}\) is bounded uniformly in \(L^{1}(\R^n)\) and for every \(\lambda \geq 1\), \(\widehat{\rho_{\lambda}}\) vanishes on a neighborhood of \(0\).

We then define a sequence \((u_{\lambda})_{\lambda\geq 1}\) in \(\mathcal{S}(\R^n,V)\) in such a way that for every \(\lambda \ge 1\) and \(\xi \in \R^n\),
\[
\widehat{u_{\lambda}}(\xi)=(-i)^k\widehat{\rho_{\lambda}}(\xi) \bigl(A(\xi)^*\circ A(\xi)\bigr)^{-1}\bigl[A(\xi)^{*}[e]\bigr].
\]
Since \(A (D)\) is elliptic and homogeneous of order \(k\), \(u_{\lambda}\) is well defined as an element of \(\mathcal{S}(\R^n,V)\) and moreover, since \(e \in \bigcap_{\xi \in \R^n \setminus \{0\}} A (\xi)[V]\), \(A(D)u_{\lambda}=\rho_{\lambda}e\). 
By a classical approximation argument and our assumption, we have for every \(\lambda \ge 1\),
\begin{equation*}
\Bigl( \int_{\R^n} \frac{\abs{D^{k - \ell} u_\lambda (x)}^q}{\abs{x}^{n - (n - \ell) q}}\dif x \Bigr)^\frac{1}{q} \le C\int_{\R^n} \abs{A(D)u_\lambda}\:.
\end{equation*}
If \(G\) is the Green function given by lemma~\ref{lemmaGreen}, this reads as 
\begin{equation}
\label{ineqrho}
\Bigl( \int_{\R^n} \frac{\abs{(D^{k - \ell} G \ast \rho_\lambda)(x)[e]}^q}{\abs{x}^{n - (n - \ell) q}}\dif x \Bigr)^\frac{1}{q} \le C\int_{\R^n} \abs{\rho_\lambda e}=C'\:.
\end{equation}

We claim that for every \(x\in \R^n\setminus \{0\}\),
\(\lim_{\lambda \to \infty} (D^{k - \ell} G\ast \rho_{\lambda})(x) = D^{k - \ell} G(x)\). 
Indeed, for every \(\lambda \ge 1\) and \(x\in \R^n\setminus \{0\},\) we write
\begin{multline}\label{eqglg}
D^{k - \ell} G*\rho_{\lambda}(x)-D^{k - \ell} G(x)\\
= \int_{\R^n}\bigl(D^{k - \ell} G(x-y)-D^{k - \ell} G(x)\bigr) \lambda^n \psi(\lambda y) \dif y - \int_{\R^n} D^{k - \ell} G(x-y) \frac{1}{\lambda^n} \psi\Bigl(\frac{y}{\lambda}\Bigr). 
\end{multline}

Since \(D^{k-\ell} G\) is smooth on \( \R^n \setminus \{0\}\) and homogeneous of degree \(-(n-\ell),\) there exists \(C>0\) such that for every \(y\in B_{\abs{x}/2},\) 
\[
\abs{D^{k - \ell} G(x-y)-D^{k - \ell} G(x)}\leq \frac{C\abs{y}}{\abs{x}^{n-\ell+1}}.
\]
Together with the fact that \(\psi\) belongs to \(\mathcal{S}(\R^n),\) this implies that for every \(\alpha \in (0, n+1)\),
\begin{multline*}
\Bigabs{\int_{\R^n}\bigl(D^{k - \ell} G(x-y)-D^{k - \ell} G(x)\bigr) \lambda^n \psi(\lambda y) \dif y }\\
\leq C'_\alpha  \int_{B_{\abs{x}/2}} \frac{\abs{y}}{\abs{x}^{n-\ell+1}} \frac{\lambda^n}{\lambda^\alpha \abs{y}^\alpha }\dif y \\ + C'_\alpha \int_{\R^n \setminus B_{\abs{x}/2}} \Bigl(\frac{1}{\abs{x-y}^{n-\ell}}+\frac{1}{\abs{x}^{n-\ell}}\Bigr)\frac{\lambda^n}{\lambda^{\alpha}\abs{y}^{\alpha}}\dif y.
\end{multline*}
This gives if \(\alpha > n\), 
\[
\Bigabs{\int_{\R^n}\bigl(D^{k - \ell} G(x-y)-D^{k - \ell} G(x)\bigr) \lambda^n \psi(\lambda y) \dif y } \leq \frac{C''}{\lambda^{\alpha - n}\abs{x}^{\alpha - \ell}}.
\]
It follows that the first term in the right hand side of \eqref{eqglg} converges to \(0\). In order to estimate the second term, we pick \(\alpha \in (\ell, n)\) and write
\[
\begin{split}
\Bigabs{\int_{\R^n}D^{k - \ell} G(x-y) \frac{1}{\lambda^n} \psi\Bigl(\frac{y}{\lambda}\Bigr) \dif y } &\leq  C'_\alpha \int_{\R^n} \frac{1}{\abs{x-y}^{n-\ell}}\frac{1}{\lambda^n} \frac{\lambda^\alpha}{\abs{y}^\alpha}\dif y\\
&\leq C' \frac{1}{\lambda^{n - \alpha}\abs{x}^{\alpha-\ell}}.
\end{split}
\]
This completes the proof of the fact that \(D^{k - \ell} G * \rho_{\lambda}\) converges pointwisely to \(D^{k - \ell} G\) on \(\R^n\setminus \{0\}\). 

By letting \(\lambda\to\infty\) in \eqref{ineqrho}, we get by Fatou's Lemma
\[
\Bigl( \int_{\R^n} \frac{\abs{D^{k - \ell} G (x)[e]}^q}{\abs{x}^{n - (n - \ell) q}}\dif x \Bigr)^\frac{1}{q} <\infty.
\]
Since \(D^{k - \ell} G\) is homogeneous of degree \(-(n - \ell)\), this implies that \(D^{k - \ell} G(x)[e] = 0\) for every \(x\ne 0\). In view of the properties of \(D^{k - \ell} G\), we thus have \(e\in \bigcap_{\xi \in \R^n \setminus \{0\}} \ker A(\xi)^{*}\). Since \(e \in \bigcap_{\xi \in \R^n \setminus \{0\}} A (\xi)[V]\), we conclude that \(e=0\) and this completes the proof of proposition~\ref{propositionNecessaryCanceling}.
\end{proof}

\section{Necessity of  the ellipticity condition}

\subsection{The Hardy--Sobolev inequality}
The necessity of the ellipticity condition in theorem~\ref{theoremCancelingHardySobolev} for the inhomogeneous inequality corresponds to the case \(p=1\) in the following proposition.
\begin{proposition}
\label{propositionNecessaryEllipticq}
Let \(A(D)\) be a linear differential operator of order \(k\) on \(\R^n\) from \(V\) to \(E\), \(p \in [1, n)\) and \(q \in (p, \frac{np}{n - p}]\).
If there exists \(C \in \R\) such that  for every \(u \in C^\infty_c(\R^n; V)\),
\[
  \Bigl( \int_{\R^n} \frac{\abs{D^{k - 1} u (x)}^q}{\abs{x}^{n - (\frac{n}{p} - 1) q }}\dif x \Bigr)^\frac{p}{q} \le C\int_{\R^n} \abs{A(D)u}^p\:,
\]
then \(A (D)\) is elliptic.
\end{proposition}
\begin{proof}
We assume by contradiction that there exist \(v \in V \setminus \{0\}\) and \(\xi \in \R^n \setminus \{0\}\) such that \(A(\xi) [v] = 0\).
Without loss of generality, we can further assume that \(\abs{\xi}=1\).
 We fix \(\varphi \in C^\infty_c(\R^n)\) and \(\psi \in C^\infty_c (\R)\) and we define for \(\lambda > 0\) the function \(u_\lambda \in C^\infty_c (\R^n; V)\) for every \(x \in \R^n\) by
\[
 u_\lambda (x) = \varphi \Bigl(\frac{x}{\lambda}\Bigr) \psi (\xi \cdot x)v\:. 
\]
By the iterated Leibniz product rule for differentiation, if \(\lambda \ge 1\) and \(x \in \R^n\),
\[
\begin{split}
\Bigabs{D^{k - 1} u_\lambda (x)-\varphi \Bigl(\frac{x}{\lambda}\Bigr) D^{k - 1}[\psi (\xi \cdot x)]v}& \leq \frac{C'}{\lambda}\Bigl(\sum_{j=1}^{k} \Bigabs{D^j\varphi \Bigl(\frac{x}{\lambda}\Bigr)} \Bigr)\Bigl(\sum_{j=0}^{k - 1} \abs{D^j\psi(\xi \cdot x)} \Bigr)\\
& \leq \frac{C'}{\lambda} \theta \Bigl(\frac{x}{\lambda}\Bigr) \eta (\xi \cdot x),
\end{split}
\]
where we have introduced to alleviate notation the functions \(\theta \in C_c (\R^n)\) and \(\eta \in C_c (\R)\) defined for \(y \in \R^n\) by \(\theta(y)=\sum_{j=1}^{k} \abs{D^j\varphi(y)}\) and for \(t \in \R\) by \(\eta(t)=\sum_{j=0}^{k - 1} \abs{D^j\psi(t)}\).
By the Minkowski inequality and our assumption, we thus get
\begin{multline}\label{eqpropnecessity1}
\Bigl( \int_{\R^n} \frac{\abs{\varphi(\frac{x}{\lambda})D^{k-1}\psi (\xi \cdot x)}^q}{\abs{x}^{n-(\frac{n}{p}-1)q}}\Bigr)^{\frac{1}{q}}\\
\leq C \Bigl( \int_{\R^n} \abs{A(D)u_\lambda}^p\Bigr)^{\frac{1}{p}} + C' \Bigl( \int_{\R^n} \frac{\abs{\theta(\frac{x}{\lambda}) \eta (\xi \cdot x)}^q}{\lambda^q \abs{x}^{n-(\frac{n}{p}-1)q}}\dif x\Bigr)^{\frac{1}{q}}.
\end{multline}
Since \(A(\xi)[v]=0\), we have for every \(x \in \R^n\),
\[
\abs{A(D)u_{\lambda} (x)}\leq \frac{C''}{\lambda}\Bigl(\sum_{j=1}^{k} \Bigabs{D^j\varphi \Bigl(\frac{x}{\lambda}\Bigr)} \Bigr)\Bigl(\sum_{j=0}^{k - 1} \abs{D^j\psi(\xi \cdot x)} \Bigr)=\frac{C''}{\lambda} \theta\Bigl(\frac{x}{\lambda}\Bigr) \eta(\xi \cdot x).
\]
If \(P_\xi\) denotes the orthogonal projection on \(\xi^\perp\) defined for \(y \in \R^n\) by \(P_\xi (y) = y - (\xi \cdot y) \xi\), we obtain by a change of variable,
\[
\begin{split}
\int_{\R^n}\abs{A(D)u_{\lambda}}^p&\leq \frac{C''^{p}}{\lambda^p} \int_{\R^n} \Bigabs{\theta\Bigl(\frac{x}{\lambda}\Bigr) \eta(\xi \cdot x)}^p \dif x\\
&= C''^{p} \lambda^{n - 1 - p} \int_{\R^n}  \Bigabs{\theta\Bigl(\frac{\xi \cdot y}{\lambda}\xi + P_\xi (y)\Bigr) \eta( \xi \cdot y)}^p \dif y.
\end{split}
\]

If we choose \(R > 0\) in such a way that \(\supp \theta \subset B_R\) and  \(\supp \eta \subset (-R,R)\), then for every \(\lambda \ge 1\) and \(y \in \R^n\),
\[
 \Bigabs{\theta\Bigl(\frac{\xi \cdot y}{\lambda}\xi + P_\xi (y)\Bigr) \eta( \xi \cdot y)}^p 
 \le \norm{\theta}_{L^{\infty}(\R^n)}^p \norm{\eta}_{L^\infty (\R)}^p\chi_{(-R,R)} (\xi\cdot y) \chi_{B_R} \bigl(P_\xi(y)\bigr) ,
\]
so that, by comparison of integrals, for every \(\lambda \ge 1\),
\begin{equation}\label{eqpropnecessity2}
\int_{\R^n}\abs{A(D)u_{\lambda}}^p \leq C''' \lambda^{n - p - 1}.
\end{equation}
By the same changes of variables, we also get for every \(\lambda \ge 1\),
\[
\int_{\R^n} \frac{\abs{\theta(\frac{x}{\lambda}) \eta (\xi \cdot x)}^q}{\lambda^q \abs{x}^{n-(\frac{n}{p}-1)q}}\dif x 
=\lambda^{(\frac{n}{p} - 1)q - 1} \int_{\R^n} \frac{\bigabs{\theta\bigl(\frac{\xi \cdot y}{\lambda}\xi + P_\xi (y)\bigr) \eta(\xi \cdot y)}^q}{\lambda^q \bigabs{\frac{\xi \cdot y}{\lambda}\xi + P_\xi (y)}^{n-(\frac{n}{p}-1)q}}\dif y.
\] 
The integrand can be bounded for every \(\alpha \in [0, n-(\frac{n}{p}-1)q]\), \(\lambda > 0\) and \(y \in \R^n\) as
\begin{multline*}
\frac{\bigabs{\theta\bigl(\frac{\xi \cdot y}{\lambda}\xi + P_\xi (y)\bigr) \eta(\xi \cdot y)}^q}{\lambda^q \bigabs{\frac{\xi \cdot y}{\lambda}\xi + P_\xi (y)}^{n-(\frac{n}{p}-1)q}} \\
\le \frac{\norm{\theta}_{L^\infty (\R^n)} \norm{\eta}_{L^\infty (\R)} \chi_{(-R, R)} (\xi \cdot y) \chi_{B_R} \bigl(P_\xi (y)\bigr)}{\lambda^{q - \alpha} \abs{\xi \cdot y}^\alpha \abs{P_\xi (y)}^{n(1 - \frac{q}{p}) + q - \alpha}}.
\end{multline*}
If \(1 - (\frac{n}{p} - 1) q < \alpha < 1\), the right-hand side is integrable, and thus
\begin{equation}\label{eqpropnecessity3}
 \int_{\R^n} \frac{\abs{\theta(\frac{x}{\lambda}) \eta (\xi \cdot x)}^q}{\lambda^q \abs{x}^{n-(\frac{n}{p}-1)q}}\dif x
\leq C'''' \lambda^{(\frac{n}{p} - 1)q - 1 - (q - \alpha)}.
\end{equation}
Finally by Fatou's lemma, we have 
\begin{multline}\label{eqpropnecessity4}
\liminf_{\lambda \to \infty} \lambda^{1-(\frac{n}{p} - 1)q} \int_{\R^n} \frac{\abs{\varphi(\frac{x}{\lambda})D^{k-1}\psi (\xi \cdot x)}^q}{\abs{x}^{n-(\frac{n}{p}-1)q}}\dif x \\
= \liminf_{\lambda \to \infty} \int_{\R^n} \frac{\bigabs{\varphi\bigl(\frac{\xi \cdot y}{\lambda}\xi + P_\xi (x)\bigr)D^{k-1}\psi (\xi \cdot y)}^q}{\bigabs{\frac{\xi \cdot y}{\lambda}\xi + P_\xi (y)}^{n-(\frac{n}{p}-1)q}}\dif y\\ 
\geq \int_{\R^n} \frac{\bigabs{\varphi\bigl(P_\xi (y)\bigr)D^{k-1}\psi (\xi \cdot y)}^q}{\bigabs{P_\xi (y)}^{n-(\frac{n}{p}-1)q}}\dif y.
\end{multline}
By inserting \eqref{eqpropnecessity2} and \eqref{eqpropnecessity3} into \eqref{eqpropnecessity1} and letting \(\lambda \to +\infty\), we deduce since \(q > p\) and \(\alpha < 1 < q\), in view of \eqref{eqpropnecessity4}
\[
\int_{\R^n} \frac{\bigabs{\varphi\bigl(P_\xi (y)\bigr)D^{k-1}\psi (\xi \cdot y)}^q}{\abs{P_\xi (y)}^{n-(\frac{n}{p}-1)q}}\dif y=0.
\]
This yields the desired contradiction because \(\psi\) and \(\varphi\) are arbitrary test functions.
\end{proof}

\subsection{The pure Hardy inequality}

In this section, we investigate the limiting case \(p=q\) in proposition~\ref{propositionNecessaryEllipticq}.

\subsubsection{Condition on the nonellipticity set}
If \(A (D)\) is not elliptic, we can consider the set of those \(\xi \in \R^n\) such that \(A (\xi)\) is not one-to-one. We show that if a Hardy inequality holds for \(A (D)\), then this set does not contain any linear subspace of dimension \(\lceil n - p \rceil\).

\begin{proposition}\label{propositionhyperplan}
Let \(A(D)\) be a linear differential operator of order \(k\) on \(\R^n\) from \(V\) to \(E\).  Let \(p\in [1,n)\).  If there exists \(C>0\) such that for every \(u\in C^{\infty}_c(\R^n ; V)\), 
\begin{equation*}
 \int_{\R^n} \frac{\abs{D^{k-1} u (x)}^p}{\abs{x}^{p}}\dif x \le C\int_{\R^n} \abs{A(D)u}^p\:,
\end{equation*}
then for every linear subspace \(\Pi\subseteq \R^n\) such that \(\dim \Pi \ge n - p\), there exists \(\xi \in \Pi\) such that \(A (\xi)\) is one-to-one.
\end{proposition}

The above proposition implies in particular that the inequality cannot hold when \(\dim V>\dim E\). It also shows that when \(n = 2\), the operator \(A(D)\) is necessarily elliptic.

In order to prove proposition~\ref{propositionhyperplan}, we rely on the following algebra property.

\begin{lemma}
\label{lemmaPotential}
Let \(A (D)\) be a homogenous differential operator from \(V\) to \(E\). Then there exists a homogeneous differential operator \(B (D) \) from \(V\) to \(V\) 
such that 
\[
  A (D) \circ B (D) = 0
\]
and 
\[
 \max_{\xi \in \R^n} \dim B (\xi)[V] = \min_{\xi \in \R^n} \dim \ker A (\xi).
\]
\end{lemma}
\begin{proof} 
Choose \(\xi_* \in \R^n\) such that \(\dim \ker A (\xi_*) = s:= \min_{\xi \in \R^n} \dim \ker A (\xi)\).
By the fundamental theorem of linear algebra, \(s = \max_{\xi \in \R^n} \dim A (\xi)^* [E]\).
Define \(P: V \to V\) to be a projection on \(A (\xi_*)^*[E]\) and choose the vectors \(e_1, \dotsc, e_s \in E\) so that their images
\(A (\xi_*)^* [e_1], \dotsc, A (\xi_*)^* [e_s]\) are linearly independent in \(V\).
Define now \(B (\xi) \in \Lin (V; V)\) by
\begin{multline*}
 B (\xi)^* [v]
 = \det \bigl(P [A (\xi)^* [e_1]], \dotsc, P [A (\xi)^* [e_s]]\bigr) v\\
 + \sum_{i = 1}^s
 (- 1)^i \det \bigl(P[v], P [A (\xi)^* [e_1]], \dotsc, P [A (\xi)^* [e_{i - 1}]], \\
 P [A (\xi)^* [e_{i + 1}]], \dotsc, P [A (\xi)^* [e_s]]\bigr) A(\xi)^*[e_i]
\end{multline*}
where \(\det\) is a determinant on \(A (\xi_*)^*[E]=P(V)\).

For every \(e_0 \in E\) and \(v \in V\), we have 
\begin{multline*}
 e_0 \cdot A (\xi)[B(\xi)[v]]
 = \sum_{i = 0}^s
 (- 1)^i \det \bigl(P [A (\xi)^* [e_0]], \dotsc, P [A (\xi)^* [e_{i - 1}]], \\
 P [A (\xi)^* [e_{i + 1}]], \dotsc, P [A (\xi)^* [e_s]]\bigr) A(\xi)^*[e_i] \cdot v.
\end{multline*}
Since the right hand side is antisymmetric with respect to the vectors \(A (\xi)^* [e_0], \dotsc, A (\xi)^* [e_s]\) and \(\dim A (\xi)^* [E] \le s\), we have \(e_0 \cdot A (\xi)[B(\xi)[v]]=0\) and, since \(e_0 \in E\) is arbitrary, \(A (\xi) [B (\xi)[v]] = 0\).
In particular, we have for every \(\xi \in \R^n\),
\(\dim B (\xi)[V] \le \dim \ker A (\xi)\). 
Since \(B (\xi_*)^*\) is one-to-one on \(\ker P\) and \(\dim \ker P = \dim \ker A (\xi_*)\), we have \(B(\xi_*)^*[V] = \ker P\) and  \(\dim B (\xi_*)[V] = \dim \ker A (\xi_*)\). 

Finally, we claim that \(P(B(\xi)^*[v])=0\) for every \(v\in V\) and every \(\xi\in \R^n\). Indeed, let \(w_0=P[v]\) and \(w_i=P(A(\xi)^*[e_i])\) for \(i\in \{1,\dotsc, s\}\). Then 
\[
P(B(\xi)^{*}[v])=\sum_{i=0}^{s} (-1)^i \det \bigl( w_0,\dots, w_{i-1}, w_{i+1}, \dots, w_s \bigr) w_i.
\]
Since \(w_i \in P[V]\) for \(i=0,\dots, s\) and \(\dim P[V]=s\), we get \(P(B(\xi)^{*}[v])=0\) which proves the claim. This implies that \(B(\xi)^*[V] \subseteq \ker P = B(\xi_*)^*[V]\). Hence,
\[
\max_{\xi \in \R^n} \dim B (\xi)^*[V] = \dim B (\xi_*)^*[V].
\]
Since \(\dim B (\xi)^*[V]=\dim B (\xi)[V],\) this completes the proof of the lemma.
\end{proof}

\begin{proof}[Proof of proposition~\ref{propositionhyperplan}]
Consider a linear subspace \(\Pi\subseteq \R^n\) such that for every \(\xi\in \Pi\), \(\rank  A(\xi) < \dim V\). 
Without loss of generality, one may assume that \(\Pi=\R^{m}\times \{0\}\). We introduce the linear differential operator on \(\R^m\) from \(V\) to \(E\) defined for \(\xi' \in \R^m\) by  \(A'(\xi'):=A(\xi',0)\). 
By lemma~\ref{lemmaPotential}, there exists a linear differential operator \(B'(D')\) on \(\R^{m}\) from \(V\) to \(V\) such that \(B'(D')\ne 0\) and \(A'(D')B'(D')=0\). 

Let \(v\in C^{\infty}_c(\R^{m};V)\) be such that \(w:=B'(D') v\ne 0\). For any \(\varphi \in C^{\infty}_c(\R^{n - m})\) and \(\lambda > 0\), we consider \(u_\lambda \in C^\infty (\R^n; V)\) defined for \((x',x'') \in \R^m\times \R^{n-m}\) by \(u_{\lambda}(x', x'') = \varphi(x'') w(\lambda x') \). Since \(A'(D')B'(D')v=0\), we get
\[
  \abs{A(D) u_{\lambda} (x', x'')} \leq C' \sum_{j=0}^{k - 1} \lambda^j \abs{D^j w(\lambda x')} \abs{D^{k - j} \varphi (x'')}\:.
\]
Whence, 
\[
 \limsup_{\lambda \to \infty} \frac{1}{\lambda^{p (k- 1) - m}} \int_{\R^n} \abs{A(D) u_{\lambda}}^p
 \leq C' \int_{\R^n} \abs{D^{k - 1} w(x')}^p \abs{D \varphi (x'')}^p\dif x'\dif x''.
\]
By definition of  \(u_\lambda,\) we have  
\[
 \lambda^{p(k-1) - m}\int_{\R^n}\frac{\abs{D^{k-1}w(x')}^p \abs{ \varphi(x'')}^p}{\frac{\abs{x'}^{p}}{\lambda^{p}}+\abs{x''}^{p}}\dif x'\dif x''
\le C'' \int_{\R^n} \frac{\abs{D^{k-1} u_\lambda (x)}^p}{\abs{x}^{p}}\dif x\:, 
\]
By the assumption applied to \(u_{\lambda},\) we thus get
\begin{multline*}
\limsup_{\lambda \to \infty}\int_{\R^n}\frac{\abs{D^{k-1}w(x')}^p \abs{ \varphi(x'')}^p}{\frac{\abs{x'}^{p}}{\lambda^{p}}+\abs{x''}^{p}}\dif x'\dif x''
\\ \le C''' \int_{\R^n} \abs{D^{k - 1} w(x')}^p \abs{D \varphi (x'')}^p\dif x'\dif x''.
\end{multline*}
We let \(\lambda\) go to \(+\infty\) and then use Fubini theorem to obtain
\[
  \int_{\R^{n-m}}\frac{\abs{\varphi(x'')}^p}{\abs{x''}^{p}}\dif x''\leq  C'''' \int_{\R^{n-m}}\abs{D \varphi (x'')}^p\dif x''.
\] 
Since this must be true for any \(\varphi \in C^{\infty}_c(\R^{n-m})\), this implies that \(p < n -  m\).
\end{proof}

\subsubsection{Dimension reduction}
In proposition~\ref{propositionhyperplan}, we dealt with failure of the ellipticity because \(A (\xi)\) is not one-to-one. The ellipticy can fail more boldly when \(A (\xi) = 0\) on a \((n-m)\)--dimensional plane. In this case, the validity of the inequality reduces to that of an inequality on the \(m\)--dimensional space. 

\begin{proposition}\label{propositionprojection}
Let \(A(D)\) be a linear differential operator of order \(k\) on \(\R^n\) from \(V\) to \(E\). We assume that there exists a vector subspace \(\Pi\subseteq \R^n\) of dimension \(m \in \{1, \dotsc, n - 1\}\) and  a linear  differential operator \(A'(D')\) of order \(k\) on \(\Pi\) from \(V\) to \(E\) such that for any \(\xi\in \R^n\), we have \(A(\xi) = A'(P(\xi))\), where \(P:\R^n\to \Pi\) is a linear projection onto \(\Pi\). Let \(\ell \in \{0, \dots, k\}\) and \(C \in \R\).\\ 
For every \(u \in C^\infty_c(\R^n; V)\),
\begin{equation*}
  \int_{\R^n} \frac{\abs{D^{k - \ell} u (x)}^p}{\abs{x}^{\ell p}}\dif x \le C\int_{\R^n} \abs{A(D)u}^p\:, 
\end{equation*}
if and only if \(k=\ell\) and for every \(u \in C^\infty_c(\Pi; V)\),
\begin{equation*}
  \int_{\Pi} \frac{\abs{u (x)}^p}{\abs{x}^{k p}}\dif x \le C \int_{\Pi} \abs{A'(D')u}^p\:. 
\end{equation*}
\end{proposition}

This proposition generalizes example \eqref{contrexempleellipticite} given in the introduction.

\begin{proof}[Proof of proposition~\ref{propositionprojection}]
Without loss of generality, we can assume that \(\Pi=\R^m\times \{0\}\) and \(P(x)=x'\) where  we write \(x=(x',x'')\in \R^m \times \R^{n-m}\). 

Assume first that the inequality on \(\R^n\) holds true. For every function \(v\in C^\infty_c(\R^m; V), w\in C^\infty_c(\R^{n-m}; V) \setminus \{0\}\) and \(\lambda > 0\), we consider the map \(u_{\lambda}(x',x'')=v(x')w(\lambda x'')\). We observe that \(A(D)u_{\lambda}(x',x'')=(A'(D')v)(x')w(\lambda x'')\) and \(\abs{D^{k - \ell}u_{\lambda}(x', x'')}\geq \lambda^{k - \ell} \abs{v(x')}\abs{D^{k - \ell}w(\lambda x'')}\). By inserting this in the inequality on \(\R^n\), we get
\begin{multline*}
\lambda^{k - \ell}\int_{\R^n}\frac{\abs{v(x')}^p\abs{D^{k - \ell}w(x'')}^p}{\bigabs{(x', x''/\lambda)}^{\ell p}}\dif x'\dif x''\\
\leq C \int_{\R^n} \abs{A'(D')v(x')}^p\abs{w(x'')}^p\dif x'\dif x''\:.
\end{multline*}
Then, necessarily, \(k=\ell\) and by letting \(\lambda\) to \(+\infty\), we get
\[
\int_{\R^n}\frac{\abs{v(x') w(x'')}^p}{\abs{x'}^{kp}}\dif x'\dif x''\leq C \int_{\R^n} \abs{A'(D')v(x')}^p\abs{w(x'')}^p\dif x'\dif x''\:,
\]
and the inequality on \(\Pi\) now follows from Fubini theorem.

Conversely, assume that \(k=\ell\) and that the inequality holds on \(\Pi\).
For every \(u \in C^\infty_c(\R^n; V)\), we have by assumption
\[
\int_{\R^m} \frac{\abs{u(x',x'')}^p}{\abs{x'}^{k p}}\dif x'\leq C \int_{\R^m}\abs{A'(D') u(x',x'')}^p\dif x'\:.
\]
It follows that
\[
\int_{\R^m} \frac{\abs{u(x',x'')}^{p}}{\abs{(x', x'')}^{k p}}\dif x'\leq C \int_{\R^m}\abs{A(D) u(x',x'')}^p\dif x'\:.
\]
We now integrate in \(x''\) to get the result.
\end{proof}

\section{Hardy inequalities for nonelliptic families of operators}

\subsection{Direct sum of directional derivatives}
We proceed to give a necessary and sufficient condition  for a special class of differential operators of order one: 

\begin{proposition}
\label{propositionHardyRankOne}
Let \(\ell \in \N_*\), \(a_1, \dotsc, a_\ell \in V \setminus\{0\}\) and \(b_1, \dotsc, b_{\ell} \in \R^n \setminus \{0\}\).
The following are equivalent
\begin{enumerate}[(i)]
\item \label{cond2dinequality} there exists \(C>0\) such that for every \(u \in C^\infty_c(\R^n; V)\),
\begin{equation*}
  \int_{\R^n} \frac{\abs{u (x)}}{\abs{x}}\dif x \le C \sum_{i = 1}^\ell \int_{\R^n} \abs{a_i \cdot D u(x)[b_i]}\:,
\end{equation*}
\item \label{cond2dalgebraic}
for every \(\xi \in \R^n \setminus \{0\}\) and \(v \in V \setminus \{0\}\) there exists \(i \in \{1, \dotsc, \ell\}\) such that 
\( (a_i \cdot v) \ne 0\) and \(\abs{b_i}^2 \xi \ne (\xi \cdot b_i) b_i\).
\end{enumerate}
\end{proposition}

If the linear differential operator \(A(D)\) of order \(1\) on \(\R^n\) from \(V\) to \(\R^\ell\) is defined for \(v \in V\) and \(\xi \in \R^n\) by
\[
A(\xi)(v):=((\xi \cdot b_i) (a_i \cdot v) )_{1\le i\le \ell},
\]
one checks that \( A (D) \) is canceling if and only if \(n \ge 2\) and that \(A (D) \) is elliptic if and only if for every \(\zeta \in \R^n \setminus \{0\}\) and \(v \in \R^n \setminus \{0\}\) there exists \(i \in \{1, \dotsc, \ell\}\) such that 
\( (a_i \cdot v) \ne 0\) and \(\zeta \cdot b_i \ne 0\), that is, instead of forbidding the vector \(b_i\) to be colinear with \(\xi\), we are prohibiting \(b_i\) from being orthogonal to \(\zeta\). In the two dimensional case, the ellipticity condition is seen to be equivalent to \eqref{cond2dalgebraic} by taking \( (\zeta \cdot \xi) = 0\), in higher dimension, the condition~\eqref{cond2dalgebraic} is weaker than the ellipticity.

For instance, when \(\ell = \dim V +1\), \eqref{cond2dalgebraic} is  satisfied if and only if \(a_1, \dots, a_{\ell}\) are \(\ell-1\) by \(\ell-1\) linearly independent and \(b_1, \dots, b_\ell\) are 2 by 2 linearly independent.

\begin{proof}
Assume by contradiction that \eqref{cond2dinequality} holds while \eqref{cond2dalgebraic} is not satisfied. Then there exist \(\xi \in \R^n \setminus \{0\}\) and \(v\in V \setminus \{0\}\) such that for every \(i \in \{1, \dotsc, \ell\}\), either \(v \cdot a_i=0\) or \(\abs{b_i}^2 \xi = (b_i \cdot \xi)b_i\). 

For every  \(\varphi\in C^{\infty}_{c}(\R)\) and \(\psi\in C^{\infty}_{c}(\R^n)\), define 
\[
  u_\lambda (x) = \varphi(\xi \cdot x)\psi \bigl(\lambda (\abs{\xi}^2 x - (\xi \cdot x) \xi)\bigr) v 
\]
and note that
\begin{multline*}
  a_i \cdot D u_\lambda (x) [b_i] = (v \cdot a_i) (\xi \cdot b_i) \varphi' (\xi \cdot x) \,\psi \bigl(\lambda (\abs{\xi}^2 x - (\xi \cdot x) \xi)\bigr)\\
  + (v \cdot a_i) \varphi (\xi \cdot x)\,\lambda D\psi \bigl(\lambda (\abs{\xi}^2 x - (\xi \cdot x) \xi)\bigr)[\abs{\xi}^2 b_i - (\xi \cdot b_i) \xi].
\end{multline*}

Now apply \eqref{cond2dinequality} to \(u_\lambda\)
\begin{multline*}
 \int_{\R^n} \frac{\bigabs{\varphi(\xi \cdot x)\psi \bigl(\lambda (\abs{\xi}^2 x - (\xi \cdot x) \xi)\bigr)}}{\abs{x}}\dif x\\
 \le C\int_{\R^n} \bigabs{\varphi' (\xi \cdot x) \psi \bigl(\lambda (\abs{\xi}^2 x - (\xi \cdot x) \xi)\bigr)} \dif x\:,
\end{multline*}
By a change of variable, this becomes 
\begin{multline*}
 \int_{\R^n} \frac{\bigabs{\varphi (\xi \cdot x)\psi \bigl(\abs{\xi}^2 x - (\xi \cdot x) \xi)\bigr)}}{\sqrt{\abs{\xi \cdot x}^2 + \lambda^{-2} (\abs{x}^2\abs{\xi}^2 - \abs{\xi \cdot x}^2)}}\dif x\\ \le C \int_{\R^n} \bigabs{\varphi' (\xi \cdot x) \psi \bigl(\abs{\xi}^2 x - (\xi \cdot x) \xi)\bigr)} \dif x\:,
\end{multline*}
By letting \(\lambda \to \infty\), we conclude that for every \(\varphi \in C^\infty_c (\R)\),
\[
  \int_{\R} \frac{\abs{\varphi (t)}}{\abs{t}}\dif t \le C\int_{\R} \abs{\varphi'(t)}\dif t,
\]
which cannot hold.

Conversely, if \eqref{cond2dalgebraic} holds true, then for every \(\xi \in \R^n \setminus\{0\}\)
\begin{equation}
\label{eqConditionH}
  \bigl\{ a_i \st i \in \{1, \dotsc, \ell\} \text{ and } \abs{b_i}^2 \xi \ne (\xi \cdot b_i) b_i \bigr\}
\end{equation}
generates \(V\). In particular, the set 
\[
  \bigl\{ a_i \st i \in \{1, \dotsc, \ell\} \bigr\}
\]
generates \(V\). 
Without loss of generality, we can assume that \(\abs{b_i} = 1\) for every \(i \in \{1, \dotsc, \ell\}\).

In order to prove the corresponding Hardy inequality \eqref{cond2dinequality}, it is thus enough to establish  that for any \(u\in C^{\infty}_c(\R^n ; V)\) and \(1\le i\le \ell\)
\[
\int_{\R^n} \frac{\abs{a_i\cdot u(x)}}{\abs{x}}\dif x \le C\sum_{j=1}^{\ell}\int_{\R^n} \abs{a_j \cdot D u (x)[b_j]}\dif x\:.
\]
Consider the case \(i = 1\) and let \(u\in C^{\infty}_c(\R^n ; V)\).
By taking \(\xi=b_1\) in \eqref{eqConditionH}, we can write \(a_1=\lambda_2 a_{i_2}+\dotsb+\lambda_r a_{i_r}\), for some \(r\in \{2,\dots, \ell\}\), \(i_j\in \{2,\dots, \ell\}\), \(\lambda_j\in \R\) and \(\abs{b_1 \cdot b_{i_j}} < 1\) for every \(j \in \{2, \dotsc, r\}\). In order to simplify the notation, we can assume that \((i_2,\dots, i_r) = (2,\dots,r)\).
It follows that
\begin{equation}\label{equation2propositionab}
a_1\cdot u(x)=\sum_{j=2}^r \lambda_j a_j\cdot u(x)\:.
\end{equation}
We now estimate
\[
  \int_{\R^n} \frac{\abs{a_1\cdot u (x)}}{\abs{x}}\dif x 
   \le \int_{\R^n \setminus \bigcup_{i = 2}^{r} B_i}  \frac{\abs{a_1\cdot u (x)}}{\abs{x}}\dif x 
       + \sum_{i=2}^{r}
             \int_{B_i}  \frac{\abs{a_1\cdot u (x)}}{\abs{x}}\dif x\:,
\]
where for \(i \in \{2, \dotsc, r\}\),
\[
  B_i = \{ x \in \R^n \st \abs{x \cdot b_1} \le \abs{x \cdot b_i}\}.
\]
By \eqref{equation2propositionab}, this gives
\begin{multline}
\int_{\R^n}\!\!\frac{\abs{a_1\cdot u (x)}}{\abs{x}}\dif x \le \int_{\R^n \setminus \bigcup_{i = 2}^{r}\!\! B_i}\sum_{j=2}^{r}\abs{\lambda_j}\frac{\abs{a_j\cdot u(x)}}{\abs{x}}\dif x + \sum_{i=2}^{r}\int_{B_i}\!\!\frac{\abs{a_1\cdot u(x)}}{\abs{x}}\dif x\\
\le C \sum_{i=2}^{r} \int_{\R^n\setminus B_i}\frac{\abs{a_i\cdot u(x)}}{\abs{x}}\dif x +\int_{B_i} \frac{\abs{a_1\cdot u(x)}}{\abs{x}}\dif x\:, \label{equation10propositionabbis}
\end{multline}
Since for every \(1\le i \le r\), the roles of \(i\) and \(1\) are symetric, we  only need to prove 
\begin{equation}\label{equation2propositionabbis} 
\int_{B_i} \frac{\abs{a_1\cdot u(x)}}{\abs{x}}\dif x\leq C \int_{\R^n} \abs{a_1\cdot Du(x)[b_1]}\dif x\:,
\end{equation}
In view of the identity
\begin{equation}\label{equation3propositionab}
a_1\cdot u(x)= \int_{-\infty}^{0}D(a_1\cdot u)(x+tb_1)[b_1]\dif t\:,
\end{equation}
we have
\begin{equation*}
\int_{B_i} \frac{\abs{a_1\cdot u (x)}}{\abs{x}}\dif x \le  \int_{B_i} \int_{-\infty}^{0}\abs{D(a_1\cdot u)(x+t b_1)[b_1]}\dif t\,\frac{\dif x}{\abs{x}} \label{equation4propositionab}\:.
\end{equation*}
We complete the proof of the proposition by the next lemma.
\end{proof}

\begin{lemma}
Let \(b, c \in \R^n \setminus \{0\}\) and define
\[
 J = \{ x \in \R^n \st \abs{b \cdot x} \le \abs{c \cdot x} \}.
\]
If \(\abs{b \cdot c} < \abs{b}^2\), then for every nonnegative function \(f \in L^1 (\R^n)\),
\[
\int_{J} \int_{\R} \frac{f (x+tb)}{\abs{x}} \dif t\dif x \leq 2\frac{\abs{b} \sqrt{\abs{b}^2 \abs{c}^2 - (b \cdot c)^2}}{\abs{b}^4 - (b \cdot c)^2} 
\int_{\R^n} f (x)\dif x .
\]
\end{lemma}

\begin{proof}
By the change of variable formula, one has 
\begin{equation}\label{chgmtvarjf}
\int_{J}\dif x \int_{\R} \frac{f (x+tb)}{\abs{x}} \dif t 
  = \int_{\R^n} \dif y \int_{D_y} \frac{f (y)}{\abs{y - t b}} \dif t,
\end{equation}
where 
\[
 D_y = \bigl\{ t \in \R \st \abs{b \cdot (y - t b)} \le \abs{c \cdot (y - t b)} \bigr\}.
\]
One notes that for every \(y \in \R^n\) and \(t \in \R\),
\[
  \abs{y - t b} \ge \Bigabs{y - \frac{b \cdot y}{\abs{b}^2} b}
\]
and that 
\[
  D_y = \Bigl\{ t \in \R \st \Bigl(\frac{(b - c) \cdot y}{\abs{b}^2 - b \cdot c} - t\Bigr) \Bigl( t - \frac{(b +c ) \cdot y}{\abs{b}^2 + b \cdot c} \Bigr) \ge 0\Bigr\},
\]
so that by the Cauchy-Schwarz inequality,
\[
\begin{split}
 \abs{D_y} &= 2\Bigabs{\frac{c \cdot (\abs{b}^2 y - (b \cdot y) b)}{\abs{b}^4 - (b \cdot c)^2}}\\
&= 2\Bigabs{\frac{(\abs{b}^2 c - (b \cdot c) b) \cdot (\abs{b}^2 y - (b \cdot y) b)}{\abs{b}^2(\abs{b}^4 - (b \cdot c)^2)}}\\
&\le 2\frac{\sqrt{\abs{b}^2\abs{c}^2 - (b \cdot c)^2}\abs{\abs{b}^2 y - (b \cdot y) b}}{\abs{b}(\abs{b}^4 - (b \cdot c)^2)}.
\end{split}
\]
In view of \eqref{chgmtvarjf}, this implies
\begin{multline*}
\int_{J} \int_{\R} \frac{f (x+tb)}{\abs{x}} \dif t  \leq \int_{\R^n}\frac{\abs{D_y}}{\Bigabs{y - \frac{b \cdot y}{\abs{b}^2} b}} f (y) \dif y \\ \leq 2\frac{\abs{b} \sqrt{\abs{b}^2 \abs{c}^2 - (b \cdot c)^2}}{\abs{b}^4 - (b \cdot c)^2} \int_{\R^n}f (y) \dif y.
\end{multline*}
This completes the proof of the lemma.
\end{proof}

\subsection{Direct sum of general differential operators}

We now generalize the sufficiency part of proposition~\ref{propositionHardyRankOne} to a more general class of non elliptic operators.
\begin{proposition}
\label{propositionHardyRankOnebis}
Consider a linear differential operator \(A(D)\) of order \(1\) from \(V\) to \(E = E_1 \oplus \dotsb \oplus E_\ell\) which can be written for \(\xi \in \R^n\) as
\[
A(\xi)= \sum_{i=1}^{\ell} A_i(P_i (\xi)) \circ Q_i,
\]
where \(P_i \in \Lin(\R^n; \R^n)\) and \(Q_i \in \Lin(V; V)\) are projections and \(A_i\) is an elliptic linear differential operator of order \(1\) on \(\Pi_i := P_i (\R^n)\) from \(V_i = Q_i (V)\) to \(E_i\).\\
If \(\bigcap_{i=1}^{\ell} \ker Q_i = \{0\}\) and for every \(i\in \{1, \dots, \ell\}\), 
\begin{equation}\label{equationvivj}
\begin{split}
\bigcap_{j \in I_i} \ker Q_j \subseteq \ker Q_i, \\
\text{ with } I_i:=\{j : \ker (P_i) \not\subseteq \ker (P_j) \textrm{ and } \ker (P_j) \not\subseteq \ker (P_i)\},
\end{split}
\end{equation}
then there exists \(C>0\) such that for every \(u \in C^\infty_c(\R^n; V)\),
\begin{equation*}
  \int_{\R^n} \frac{\abs{u (x)}}{\abs{x}}\dif x \le C \sum_{i = 1}^\ell \int_{\R^n} \abs{A(D) u (x)}\dif x\:.
\end{equation*}
\end{proposition}

The assumption \eqref{equationvivj} implies that \(A(D)\) is canceling. Indeed, for every \(i\in \{1, \dots, \ell\},\) either \(V_i=\{0\}\) or \(I_i\not=\emptyset\). In the latter case, \(\ker (P_i)\ne \{0\}\). Hence, there exists \(\xi\not= 0\) such that \(P_i(\xi)=0,\) which implies that \(A(\xi)[V]\cap E_i =\{0\}\). If \(V_i=\{0\}\), this is true for any \(\xi \in \R^n\).  Since this holds for every \(i\in \{1, \dots, \ell\},\)  \(A(D)\) is canceling.

This proposition coincides with proposition~\ref{propositionHardyRankOne} in the particular case when \(P_i(\xi)= (\xi \cdot b_i) b_i\) and \(Q_i(v)=(a_i\cdot v) a_i\) for  \(a_1, \dotsc, a_\ell \in V \setminus\{0\}\) and \(b_1, \dotsc, b_{\ell} \in \R^n \setminus \{0\}\). We thus have \(V_i = \R a_i\)  and \(\Pi_i = \R b_i\). Indeed, the assumptions \(\bigcap_{i=1}^{\ell} \ker Q_i = \{0\}\)  and \(\bigcap_{j \in I_i} \ker Q_j \subseteq \ker Q_i\) hold true  if and only if the families \(\{a_i\}_{1\leq i \leq \ell}, \{b_i\}_{1\leq i \leq \ell}\) satisfy assumption \eqref{cond2dalgebraic} in proposition~\ref{propositionHardyRankOne}.

As an example, consider the linear differential operator \(A(D)\) on \(\R^4\) from \(\R^2\) to \(\R^4\) defined for \(\xi=(\xi_1,\dots, \xi_4)\) by
\[
A(\xi) = \left( \begin{array}{cc} \xi_1 & 0 \\ 0 & \xi_2\\ \xi_3 & -\xi_4 \\ \xi_4 & \xi_3 \end{array} \right).
\] 
This operator, which is not elliptic and cannot be considered in the framework of proposition~\ref{propositionHardyRankOne}, has the form described in proposition~\ref{propositionHardyRankOnebis} with \(V_1=\R (1,0), V_2 =\R (0,1), V_3=\R^2\) and
\begin{gather*}
A_1(\xi_1) \circ Q_1= (\xi_1\quad 0), \quad  A_2(\xi_2) \circ Q_2 = (0 \quad \xi_2), \\
A_3(\xi_3, \xi_4) \circ Q_3 =\left(\begin{array}{cc} \xi_3 & -\xi_4\\ \xi_4 & \xi_3 \end{array}\right).
\end{gather*}
Assumption \eqref{equationvivj} is satisfied so that the Hardy inequality holds true in that case.
\begin{proof}[Proof of proposition~\ref{propositionHardyRankOnebis}]
The proof is very similar to  the proof of proposition~\ref{propositionHardyRankOne}. We only outline the main differences. Without loss of generality, we assume that \(V_i\not= \{0\}\) and \(\Pi_i\not= \{0\}\).
If \(\Tilde{P}_i\) is the orthogonal projection on \(\ker (P_i)^\perp\), then there exists a differential operator   \(\Tilde{A}_i \) of order \(1\) on \( \ker (P_i)^\perp\) from \(V_i\) to \(E_i\) such that \(\Tilde{A}_i \circ \Tilde{P}_i = A_i \circ P_i\). By construction \(\ker \Tilde{P}_i = \ker P_i\) and \(\Tilde{A}_i\) is elliptic. We can thus assume without loss of generality that \(P_i\) is an orthogonal projection.

Since \(\bigcap_{i=1}^{\ell} \ker Q_i = \{0\}\), there exists \(C>0\) such that for every \(v \in V\), we have 
\[
\abs{v} \leq C \sum_{i=1}^{\ell} \abs{Q_i(v)}.
\]
Thus, we only need to prove that for every \(i \in \{1, \dotsc, \ell\}\), 
\begin{equation}\label{equationchoicei1}
\int_{\R^n} \frac{\abs{Q_i(u)(x)}}{\abs{x}}\dif x \le C \sum_{j=1}^{\ell} \int_{\R^n} \abs{A_j(P_j (D)) Q_j(u)(x)}\dif x\:,
\end{equation}

Consider the case \(i=1\). We define for \(j\in I_1\) the set \(B_j:= \{x \in \R^n : \abs{P_1(x)} \leq \abs{P_j(x)}\}\). Since \(\bigcap_{j\in I_1} \ker Q_j \subseteq \ker Q_1\), there exists \(C>0\) such that for every \(v\in V\),
\[
\abs{Q_1(v)} \leq C \sum_{j\in I_1} \abs{Q_j (v)}.
\]
By using this estimate on \(\R^n\setminus \bigcup_{j\in I_1} B_j\) exactly as in proposition~\ref{propositionHardyRankOne}  (see \eqref{equation10propositionabbis}) we are thus reduced to prove  the analogue of \eqref{equation2propositionabbis}, namely for every \(j \in I_1\)
\begin{equation}
\int_{B_j} \frac{\abs{Q_1(u)(x)}}{\abs{x}}\dif x \leq C \int_{\R^n} \abs{A_1(P_1 (D) )Q_1(u)(x)}\dif x \: .
\end{equation}

Let \(n_1= \dim \Pi_1\). Consider first the case \(n_1=1\) : there exists \(b_1 \in \R^n\), \(|b_1|=1\) and a linear map \(a_1 \in \Lin(V_1 ; E_1) \) such that for every \(v\in V_1\),
\[
A_{1}(P(\xi))[v] = \xi \cdot b_1 a_1[v].
\]
Since \(a_1\) is one-to-one, there exists \(C>0\) (not depending on \(u\) ) such that for every \(x\in \R^n\)
\[
\abs{Q_1(u)(x)} \leq C \abs{a_1[Q_1(u)(x)]}
\] 
By the identity \eqref{equation3propositionab} applied to \(a_1(Q_1(u))\), we thus get
\[
\int_{B_j}\frac{\abs{Q_1(u)(x)}}{\abs{x}}\dif x \leq C \int_{B_j} \frac{\dif x}{\abs{x}}\int_{-\infty}^{0} \abs{A_1(P_1(D))(Q_1(u))(x+t b_1)}\dif t.
\]
By the change of variable formula, we get
\begin{multline}
\int_{B_j}\frac{\abs{Q_1(u)(x)}}{\abs{x}}\dif x  \leq C \int_{\R^n}\abs{A_1(P_1(D))(Q_1(u))(y)}\dif y\int_{J^j_j} \frac{dt}{\abs{y-t b_1}} \label{equationcasn1=1}
\end{multline}
where
\[
J^y_j=\{t : \abs{P_1(y-t b_1)}\leq \abs{P_j(y-t b_1)}\}.
\]

When \(n_1 \geq 2\), we use the Green function \(G_1\) corresponding to \(A_1(P_1(D))\) on \(\Pi_1\) given by lemma~\ref{lemmaGreen}, for which \(G_1\) is homogeneous of degree \(1-n_1\). We write every \(x\in \R^n\) as \(x=(y, z)\in \Pi_1\times \ker P_1\):
\begin{multline}
\label{equationcasn1g1}
\int_{B_j} \frac{\abs{Q_1(u)(x)}}{\abs{x}}\dif x  \\ \leq C \int_{B_j}\frac{\dif y dz}{\abs{(y,z)}}\int_{\Pi_1}\frac{1}{\abs{y-t}^{n_1-1}}\abs{A_1(P_1(D)) Q_1(u)(t,z)}\dif t \\ 
\leq C \int_{\R^n}\abs{A_1(P_1(D)) Q_1(u)(t,z)}\dif t \dif z\int_{B_j^{z}} \frac{\dif y}{\abs{(y,z)}\abs{y-t}^{n_1-1}},
\end{multline}
where 
\[
 B_j^{z}=\{y \in \Pi_1 : (y, z) \in B_j\}.
\]
In view of \eqref{equationcasn1=1} and \eqref{equationcasn1g1},  proposition \ref{propositionHardyRankOnebis} then follows from the next lemma.
\end{proof}
\begin{lemma}\label{lemmafinal917}
There exists \(C>0\) such that for every \(t\in \Pi_1\), for every \(z\in \ker P_1\), 
\begin{equation}\label{equationlast}
\int_{B_j^{z}} \frac{\dif y}{\abs{(y,z)}\abs{y-t}^{n_1-1}} \le C.
\end{equation}
\end{lemma}
\begin{proof} We write \(\Pi_1=(\Pi_1\cap \Pi_j)\oplus \Pi_{1}'\), where \(\Pi_{1}'=\ker P_j \cap \Pi_1\), and any \(y\in \Pi_1\) as \(y=y' + y'' \in \Pi_{1}' \oplus (\Pi_1\cap \Pi_j)\). We thus have 
\[
\abs{P_1(y, z)}^2= \abs{y}^2 = \abs{y'}^2 + \abs{y''}^2. 
\] 
Since \(y''\in \Pi_1\cap \Pi_j\), \(z\in \ker P_1\) and \(P_j\) is an orthogonal projection, we have \(y'' \cdot P_{j}(z)=0\). This gives
\[
\abs{P_j(y, z)}^2= \abs{y'' + P_{j}(z)}^2 = \abs{y''}^2 + \abs{P_{j}(z)}^2.
\]
Hence, the set \(B_j^z\) is a cylinder:
\[
  B_j^{z} = \{ y \in \Pi_1 \st \abs{y'} \leq \abs{P_{j}(z)} \}.
\]
By a pointwise bound on the integrand, we have
\[
\int_{B_j^{z}} \frac{\dif y}{\abs{(y,z)}\abs{y-t}^{n_1-1}}\\
\le 
\int_{B_j^{z}} \frac{\dif y}{\abs{z}^{1/2} \abs{y}^{1/2}\abs{y-t}^{n_1-1}}.
\]
By a double application of the Hardy--Littlewood rearrangement inequality (see for example \cite{LiebLoss}*{theorem 3.4}),
\[
\begin{split}
\int_{B_j^{z}} \frac{\dif y}{\abs{z}^{1/2} \abs{y}^{1/2}\abs{y-t}^{n_1-1}}
&\le \int_{B_j^{z}} \frac{\dif y}{\abs{z}^{1/2} \abs{y}^{1/2}\abs{(y',y''-t'')}^{n_1-1}}\\
&\le \int_{B_j^{z}} \frac{\dif y}{\abs{z}^{1/2} \abs{y}^{n_1 - 1/2}}.
\end{split}
\]
Since \(\Pi_1\not\subseteq \Pi_j\), we have \(\dim \Pi_1\cap \Pi_j < n_1\), so that the right-hand side integral is finite.
By homogeneity, we thus have
\[
  \int_{B_j^{z}} \frac{\dif y}{\abs{z}^{1/2} \abs{y}^{n_1 - 1/2}}
  = C' \frac{\abs{P_j (z)}^{1/2}}{\abs{z}^{1/2}}.
\]
and the conclusion follows.
\end{proof}

\end{document}